\documentclass[reqno,12pt]{amsart}
\usepackage[margin=1in]{geometry}
\usepackage{amsfonts}
\usepackage[all]{xy}
\usepackage{enumitem}
\usepackage{mathtools}
\usepackage{amssymb,amsmath,url}
\usepackage{setspace}
\usepackage{amsthm}
\usepackage{cite}
\usepackage{color}
\numberwithin{equation}{section}
\usepackage[normalem]{ulem}
\usepackage{graphicx}
\usepackage{lmodern}
\graphicspath{ {images/} }
\usepackage{pdfpages}
\usepackage{mmacells}
\mmaDefineMathReplacement[≤]{<=}{\leq}
\mmaDefineMathReplacement[≥]{>=}{\geq}
\mmaDefineMathReplacement[≠]{!=}{\neq}
\mmaDefineMathReplacement[→]{->}{\to}[2]
\mmaDefineMathReplacement[⧴]{:>}{:\hspace{-.2em}\to}[2]
\mmaDefineMathReplacement{∉}{\notin}
\mmaDefineMathReplacement{∞}{\infty}
\mmaDefineMathReplacement{𝕕}{\mathbbm{d}}
\mmaSet{
	morefv={gobble=2},
	linklocaluri=mma/symbol/definition:#1,
	morecellgraphics={yoffset=1.9ex}
}

\author{Zhenghui Huo and Brett D. Wick}
\title[Weak-type estimates for the Bergman projection]{Weak-type estimates for the Bergman projection on the polydisc and the Hartogs triangle}
\begin{document}
	\thanks{BDW's research is partially supported by National Science Foundation grants DMS \# 1560955 and DMS \# 1800057.}
\address{Zhenghui Huo, Department of Mathematics and Statistics, The University of Toledo,  Toledo, OH 43606-3390, USA}
\email{zhenghui.huo@utoledo.edu}
\address{Brett D. Wick, Department of Mathematics and Statistics, Washington University in St. Louis,  St. Louis, MO 63130-4899, USA}
\email{wick@math.wustl.edu}
		\newtheorem{thm}{Theorem}[section]
	\newtheorem{cl}[thm]{Claim}
	\newtheorem{lem}[thm]{Lemma}
		\newtheorem{ex}[thm]{Example}
	\newtheorem{de}[thm]{Definition}
	\newtheorem{co}[thm]{Corollary}
	\newtheorem*{thm*}{Theorem}
	\theoremstyle{definition}
		\newtheorem{rmk}[thm]{Remark}

	\maketitle
\begin{abstract}
In this paper, we investigate the weak-type regularity of the Bergman projection. The two domains we focus on are the polydisc and the Hartogs triangle. 

For the polydisc we provide a proof that the weak-type behavior is of ``$L\log L$'' type.  This result is likely known to the experts, but does not appear to be in the literature.

For the Hartogs triangle we show that the operator is of weak-type $(4,4)$; settling the question of the behavior of the projection at this endpoint.  At the other endpoint of interest, we show that the Bergman projection is \textit{not} of weak-type $(\frac{4}{3}, \frac{4}{3})$ and provide evidence as to what the correct behavior at this endpoint might be.

\medskip

\noindent
{\bf AMS Classification Numbers}:  32A07, 32A25, 32A36

\medskip

\end{abstract}

\section{Introduction}
Let $\Omega\subseteq \mathbb C^n$ be a bounded domain. Let $L^2(\Omega)$ denote the space of square-integrable functions with respect to the Lebesgue measure $dV$ on $\Omega$.  Let $A^2(\Omega)$ denote the subspace of square-integrable holomorphic functions. The Bergman projection $P$ is the orthogonal projection from $L^2(\Omega)$ onto $A^2(\Omega)$. Associated with $P$, there is a unique function $K_\Omega$ on $\Omega\times\Omega$ such that for any $f\in L^2(\Omega)$:
\begin{equation}
P(f)(z)=\int_{\Omega}K_\Omega(z;\bar w)f(w)dV(w).
\end{equation}
Let $P^+$ denote the absolute Bergman projection defined by:
\begin{equation}
P^+(f)(z)=\int_{\Omega}|K_\Omega(z;\bar w)|f(w)dV(w).
\end{equation}
 Mapping properties of $P$ have been an object of considerable interest for many years.  By its definition, the Bergman projection is a $L^2$ bounded operator. It is natural to consider the regularity of $P$ in other settings.
Using known estimates for the Bergman kernel, $L^p$ regularity results have been obtained  in various settings. See \cite{Fefferman,PS,McNeal1,McNeal3,NRSW,McNeal3,McNeal2,MS,CD,EL,BS,Zhenghui2}. In all these results, the domain needs to satisfy certain nice boundary conditions. On some other domains, the projection has only a finite range of mapping regularity.  See for example \cite{KP,BS,Yunus,DebrajY,EM,EM2,CHEN} for recent progress along this line.  

Among the results mentioned above, there are mainly two techniques adopted. One is to use the Schur's test (see for example \cite{Zhu}), where boundedness can be deduced from analyzing the behavior of the absolute Bergman projection on a certain test function $h$. In many cases, one can choose $h$ to be the distance function to the boundary of the domain $\Omega$ or the Bergman kernel on the diagonal.
The second approach is to show that the (absolute) Bergman projection satisfies certain weak-type estimate. For example, if the projection is of weak-type $(1,1)$, then its $L^2$ regularity together with interpolation theorem implies the $L^p$ regularity for $1<p\leq 2$. Since the Bergman projection is self-adjoint,  $L^p$ regularity for $1<p\leq 2$ yields $L^p$ regularity for $1<p<\infty$. 

While both techniques are powerful tools on obtaining $L^p$ regularity results, the Schur's test is unable to tell the weak-type regularity of the operator near the endpoint of its $L^p$ range.
In this paper, we choose the polydisc and the Hartogs triangle as two classical examples and investigate the weak-type regularity of the Bergman projection on them. 

The polydisc serves as a simple example where the Bergman kernel function is of a product form. It is well known that the weak-type behavior of the classical operators in the multi-parameter setting could be very different from the one-parameter case. For instance, the double Hilbert transform $H_1H_2$ on $\mathbb R^2$ and the Hilbert transform $H$ on $\mathbb R$ behave differently near $L^1$: $H$ is of weak-type (1,1) while $H_1H_2$ is of weak-type $L\log^+L$. See for example \cite{Fefferman72}. By the same reason, one should expect the weak-type regularity of the Bergman projection on $\mathbb D^n$ to be different from the Bergman projection on $\mathbb D$. 

The Hartogs triangle $\mathbb H$, on the other hand, is a classical model where the projection has only limited $L^p$ regularity.  It was shown by Chakrabarti and Zeytuncu in \cite{DebrajY} that the Bergman projection on the Hartogs triangle is $L^p$-regular if and only if $\frac{4}{3}<p<4$. Since the Hartogs triangle is biholomorphically equivalent to $\mathbb D\times \mathbb D\backslash\{0\}$, the $L^p$ boundedness of the Bergman projection on $\mathbb H$ can also be related to the regularity of the projection on the weighted space $L^p(\mathbb D^2,|z_2|^{2-p})$. From this perspective, both the product structure of $\mathbb D^2$ and the weight $|z_2|^{2-p}$ may affect the weak-type regularity of the projection near $L^{\frac{4}{3}}$ and $L^4$.

We summarize our results about the Bergman projection $P$ as follows:
\begin{enumerate}
	\item On the bidisc, $P$ is not of weak-type (1,1). (Theorem 3.1)
		\item On the polydisc $\mathbb D^n$, $P$  is of weak-type $L(\log^+L)^{n-1}$. (Theorem 3.7)
	\item On the Hartogs triangle $\mathbb H$, $P$ is \textit{not} of weak-type $({4}/{3},{4}/{3})$. (Theorem 4.1)
	\item On the Hartogs triangle $\mathbb H$, $P$  is of weak-type $(4,4)$. (Theorem 4.2)
	\item For any $\epsilon>0$, $P$ on $\mathbb H$ is bounded from $L^{\frac{4}{3}}(\mathbb H,|z_2|^{-\epsilon})$ to $L^{\frac{4}{3},\infty}(\mathbb H)$. (Theorem 4.6)
\end{enumerate}
Results (1)  and (2) above are not surprising from a multi-parameter analysis perspective, and hence could be known to people. Since we couldn't find them in the literature, we decide to put them here. As a consequence of Result (4), the projection $P$ is  bounded from the Lorentz space $L^{4/3,1}(\mathbb H)$ to $L^{4/3}(\mathbb H)$. See Remark 4.4. Also, we provide refinements of Result (5). See Theorems 4.7 and 4.9. 

 The paper is organized as follows: In Section 2, we recall the definition of the Hartogs triangle and provide lemmas that will be used in the paper. In Section 3, we consider weak-type estimates for the Bergman projection on the bidsc and the polydisc. We give an example in the proof of Theorem 3.1 to show that the projection $P$ is not of weak-type (1,1). We show in Theorem 3.7 that $P$ on the polydisc $\mathbb D^n$ is of weak-type $L(\log^+L)^{n-1}$. In Section 4, we state and prove weak-type results for the Bergman projection on the Hartogs triangle.

\section{Preliminaries}
 Let $\mathbb D$ denote the unit disc in $\mathbb C$. The Bergman kernel on $\mathbb D$ is given by,
 \begin{equation}
 K_{\mathbb D}(z;\bar w)=\frac{1}{\pi(1-z\bar w)^2}, \;\;\;\text{ for } z,w\in \mathbb D.
 \end{equation}
 Since the Bergman kernel on the product domain $\Omega_1\times\Omega_2$ equals the product of the kernel functions on $\Omega_1$ and $\Omega_2$, the Bergman kernel on the polydisc $\mathbb D^n$ is given by,
 \begin{equation}
 K_{\mathbb D^n}(z;\bar w)=\prod_{j=1}^{n}\frac{1}{\pi(1-z_j\bar w_j)^2},\;\;\; \text{ for } z,w\in \mathbb D^n.
 \end{equation}
 The Hartogs triangle $\mathbb H$ is defined by $\mathbb H=\{(z_1,z_2)\in \mathbb C^2:|z_1|<|z_2|<1\}.$
Let $\mathbb D^*$ denote the punctured disc $\mathbb D\backslash \{0\}$. 
The mapping $(z_1,z_2)\mapsto (\frac{z_1}{z_2},z_2)$ is a biholomorphism from $\mathbb H$ onto $\mathbb D\times \mathbb D^*$. The biholomorphic transformation formula (see \cite{Krantz}) then implies that 
	\begin{align}\label{K}
K_{\mathbb H}(z_1,z_2;\bar w_1,\bar w_2)&=\frac{1}{z_2\bar w_2}K_{\mathbb D\times\mathbb D^*}\left(\frac{z_1}{z_2},z_2;\frac{\bar w_1}{\bar w_2},\bar w_2\right)\nonumber\\
&=\frac{1}{z_2\bar w_2}K_{\mathbb D\times\mathbb D}\left(\frac{z_1}{z_2},z_2;\frac{\bar w_1}{\bar w_2},\bar w_2\right)\nonumber\\
&=\frac{1}{\pi^2z_2\bar w_2(1-\frac{z_1\bar w_1}{z_2 \bar w_2})^2(1-z_2\bar w_2)^2}.
\end{align}	
The second equality sign above holds since $A^2(\mathbb D\times\mathbb D^*)$ and $A^2(\mathbb D^2)$ are identical.

	Given functions of several variables $f$ and $g$, we use $f\lesssim g$ to denote that $f\leq Cg$ for a constant $C$. If $f\lesssim g$ and $g\lesssim f$, then we say $f$ is comparable to $g$ and write $f\approx g$. We reference below the Forelli-Rudin estimate. See for example \cite{Zhu} for its proof.
\begin{lem}[Forelli-Rudin]\label{lemma2.1} Let $\sigma$ denote Lebesgue measure on the unit circle $\mathbb S^1\subset\mathbb C$.
	For $\epsilon<1$ and $w\in\mathbb D$, let 
	\begin{equation}\label{**}
	a_{\epsilon,\delta}(w)=\int_{\mathbb D}\frac{(1-|\eta|^2)^{-\epsilon}}{|1-w\bar\eta |^{2-\epsilon-\delta}}dV(\eta),
	\end{equation}
	and let
	\begin{equation}\label{ }
	b_\delta(w)=\int_{\mathbb S^1}\frac{1}{|1-w\bar\eta|^{1-\delta}}d\sigma(\eta).
	\end{equation}
	Then \begin{enumerate}
		\item for $\delta>0$, both $a_{\epsilon,\delta}$ and $b_{\delta}$ are bounded on $\mathbb D$;
		\item for $\delta=0$, both $a_{\epsilon,\delta}(w)$ and $b_{\delta}(w)$ are comparable to the function $-\log(1-|w|^2)$;
		\item for $\delta<0$, both $a_{\epsilon,\delta}(w)$ and $b_{\delta}(w)$ are comparable to the function $(1-|w|^2)^{\delta}$.
	\end{enumerate}
\end{lem}
We also recall the weighted inequalities by Bekoll\'e and Bonami in \cite{BB78} for $P$ and $P^+$ on the unit disk:
\begin{lem}[Bekoll\'e-Bonami]
 Let $T_z$ denote the Carleson tent over $z$ in $\mathbb D$ defined as below:
\begin{itemize}
	\item $T_z:=\left\{w\in \mathbb D:\left|1-\bar w\frac{z}{|z|}\right|<1-|z|\right\}$ for $z\neq 0$, and
	\item $T_z:= \mathbb D$ for $z=0$.
\end{itemize} 	Let the weight $\mu$ be a positive, locally integrable function on the $\mathbb D$. Let $1<p<\infty$. Then the following conditions are equivalent:
	\begin{enumerate}
		\item $P:L^p(\mathbb D,\mu)\mapsto L^p(\mathbb D,\mu)$ is bounded;
		\item $P^+:L^p(\mathbb D,\mu)\mapsto L^p(\mathbb D,\mu)$ is bounded;
		\item The Bekoll\'e-Bonami constant $B_p(\mu)$ is finite where: $$B_p(\mu):=\sup_{z\in \mathbb D}\frac{\int_{T_z}\mu(w) dV(w)}{\int_{T_z}dV(w)}\left(\frac{\int_{T_z}\mu^{-\frac{1}{p-1}} (w)dV(w)}{\int_{T_z}dV(w)}\right)^{p-1}.$$
	\end{enumerate} 
\end{lem}
We end this section by recalling the definitions of the weak $L^p$ space, weak-type $(p,p)$, and the $L(\log^+L)^k$ space. Given a subset $U$ in the domain $\Omega$ and let $\mu$ be a measure on $\Omega$. We use the notation $\mu(U)$ to denote the $\mu$-measure of $U$. When $\mu$ is the Lebesgue measure, we will simply write $|U|$. 
\begin{de}
Let $(X,\mu)$ be a measure space. For $0<p<\infty$, the weak $L^p$ space $L^{p,\infty}(X,\mu)$ is defined as the set of all $\mu$-measurable functions $f$ such that
\begin{align}
\|f\|_{L^{p,\infty}}=\inf\left\{C>0:\mu\{x\in X:|f(x)|>\lambda\}\leq\frac{C^p}{\lambda^p} \;\;\text{ for all } \lambda>0\right\}<\infty.
\end{align}

\end{de}
\begin{de}
	Let $(X,\mu)$ and $(Y,\nu)$ be two measure spaces. Let $0<p<\infty$ and $0<q<\infty$. An operator $T$ that is said to be of weak-type $(p,q)$ if $T$ is bounded from $L^p(X,\mu)$ to $L^{q,\infty}(Y,\nu)$, i.e. for any $f\in L^p(X,\mu)$ and any $\lambda>0$, 
	\begin{equation}\label{2.8}
	\nu(\{y\in Y:|T(f)(y)|>\lambda\})\lesssim\frac{\|f\|^q_{L^p(X,\mu)}}{\lambda^q}.
	\end{equation}\end{de}

\begin{de}
Set $\log^+ x:=\begin{cases}0& x=0\\\max\{0,\log x\}& x>0.\end{cases}$ Let $\mathcal L^p(\log^+\mathcal L)^k(\Omega)$ be the set of all functions $f$ on $\Omega$ satisfying $\int_{\Omega} |f(z)|^p(\log^+|f(z)|)^kdV<\infty$. We define the Orlicz space $L^p(\log^+L)^k(\Omega)$ to be the linear hull of  $\mathcal L^p(\log^+\mathcal L)^k(\Omega)$ with the norm
$$\|f\|_{L^p(\log^+L)^k(\Omega)}=\inf\left\{\lambda>0:\int_{\Omega} \left|{f(z)}/{\lambda}\right|^p\left(\log^+\left|{f(z)}/{\lambda}\right|\right)^kdV(z)\leq 1\right\}.$$
We say an operator $T$ is of weak-type $L^p(\log^+L)^k$ on $\Omega$ if for any $f\in L^p(\log^+L)^k(\Omega)$ and any $\lambda>0$,
\begin{equation}\label{2.9}|\{z\in \Omega:|T(f)(z)|>\lambda\}|\lesssim \frac{\|f\|^p_{L^p(\log^+L)^k(\Omega)}}{\lambda^p}.\end{equation}
\end{de}
For more details about the Orlicz space, see for example \cite{Rao}. 
\begin{rmk}It is worth noting that when $\nu$ is a finite measure and $\lambda$ is chosen to be small, inequalities (\ref{2.8}) and (\ref{2.9}) trivially holds. Since all the domains involved in this paper are bounded and hence have finite Lebesgue measure, we only need to check (\ref{2.8}) for large $\lambda$ to prove the weak-type results.\end{rmk}
\section{Weak-type estimates for the Bergman projection on the polydisc} 
The results in this section are not surprising from a multi-parameter analysis perspective, and could be known to experts. Since we couldn't find them in the literature, we decide to put them here with their proofs.
\begin{thm}
	The Bergman projection $P$ on the bidisc $\mathbb D^2$ is not of weak-type $(1,1)$.
\end{thm}
\begin{proof}
By (\ref{2.8}),	it suffices to show that there exists a parameter family of integrable functions $\{f_s\}$ on $\mathbb D^2$ satisfying the inequality below:
	\begin{equation}\label{3.1}
	|\{(z_1,z_2)\in \mathbb D^2:|P(f_s)(z_1,z_2)|>\lambda\}|\geq\frac{C_s\|f_s\|_{L^1}}{\lambda},
	\end{equation}
	where $C_s$ can be arbitrarily large.
	For $1>s>0$, we set $$f_s(w)=(1-s^2)^4|1-sw_1|^{-4}|1-sw_2|^{-4}={\pi^4(1-s^2)^4}|K_{\mathbb D^2}(s,s;\bar w_1, \bar w_2)|^2.$$ Hence  $\|f_s\|_{L^1}=\pi^4(1-s^2)^4K_{\mathbb D^2}(s,s;s,s)=\pi^2$. On the other hand, it is easy to see that \begin{align}P(f_s)(z_1,z_2)&=\int_{\mathbb D^2}\frac{(1-s^2)^4(1-s\bar w_1)^{-2}(1-s\bar w_2)^{-2}}{\pi^2(1-z_1\bar w_1)^2(1-z_2\bar w_2)^2}\frac{1}{(1-sw_1)^2(1-sw_2)^2}dV(w_1,w_2)\nonumber\\&=\overline{P\left(\frac{(1-s^2)^4(1-sw_1)^{-2}(1-sw_2)^{-2}}{(1-\bar z_1 w_1)^2(1-\bar z_2 w_2)^2}\right)(s,s)}\nonumber\\&=\frac{(1-s^2)^4(1-s^2)^{-2}(1-s^2)^{-2}}{(1-sz_1)^2(1-sz_2)^2}=(1-sz_1)^{-2}(1-sz_2)^{-2}.\end{align}
	Therefore
	 $\{(z_1,z_2)\in \mathbb D^2:|P(f_s)(z_1,z_2)|>\lambda\}=\left\{(z_1,z_2)\in \mathbb D^2:|1-sz_1|^{-2}|1-sz_2|^{-2}>\lambda\right\}.$
	Set $U_{\lambda,s}=\left\{(z_1,z_2)\in \mathbb D^2:|1-sz_2|^2<|1-sz_1|^{-2}\lambda^{-1} \text { and } 2|1-sz_1|<(1-s)^{-1}\lambda^{-1/2}\right\}$. Then 
\begin{align}\label{3.2}
&\left|\left\{(z_1,z_2)\in \mathbb D^2:|1-sz_1|^{-2}|1-sz_2|^{-2}>\lambda\right\}\right|\nonumber\\\geq&|U_{t,s}|=
\int_{\left\{z_1\in\mathbb D:2|1-sz_1|<(1-s)^{-1}\lambda^{-1/2}\right\}}\int_{\left\{z_2\in \mathbb D:|1-sz_2|<{|1-sz_1|^{-1}\lambda^{-1/2}}\right\}}dV(z_2)dV(z_1).
\end{align}
By a change of the variable $z_2=\frac{i-w_2}{i+w_2}$, we have 
\begin{align}
&\int_{\left\{z_2\in\mathbb D:|1-sz_2|<{|1-sz_1|^{-1}}\lambda^{-1/2}\right\}}dV(z_2)\nonumber\\=&\int_{\left\{w_2\in \mathbb C:\text{Im}(w_2)>0,\frac{|(1-s)i+(1+s)w_2|}{|i+w_2|}<\frac{1}{|1-sz_1|\lambda^{1/2}}\right\}}\frac{4}{|i+w_2|^4}dV(w_2).
\end{align}
When $|w_2|<1$, we have $|i+w|\approx 1$ and $|(1-s)i+(1+s)w_2|\leq (1+s)|w_2|+(1-s)$. Combining these inequalities with the fact that $2|1-sz_1|<(1-s)^{-1}\lambda^{-1/2}$ for $(z_1,z_2)\in U_{t,s}$, there holds
\begin{align}\label{3.4}
&\int_{\left\{w_2\in \mathbb C:\text{Im}(w_2)>0,\frac{|(1-s)i+(1+s)w_2|}{|i+w_2|}<\frac{1}{|1-sz_1|\lambda^{1/2}}\right\}}\frac{4}{|i+w_2|^4}dV(w_2)\nonumber\\\geq
&\int_{\left\{w_2\in \mathbb C:|w_2|<1,\text{Im}(w_2)>0,\frac{|(1-s)i+(1+s)w_2|}{|i+w_2|}<\frac{1}{|1-sz_1|\lambda^{1/2}}\right\}}\frac{4}{|i+w_2|^4}dV(w_2)\nonumber\\\gtrsim& \int_{\left\{w_2\in \mathbb C:|w_2|<1,\text{Im}(w_2)>0,{|w_2|}<\frac{1}{|1-sz_1|\lambda^{1/2}}-(1-s)\right\}}dV(w_2)\nonumber\\\gtrsim& \left(\frac{1}{|1-sz_1|\lambda^{1/2}}-(1-s)\right)^2\gtrsim {|1-sz_1|^{-2}}\lambda^{-1}.
\end{align}
Applying inequalities (\ref{3.4}) and Lemma 2.1 to (\ref{3.2}) and choose $\lambda=16^{-1}(1-s)^{-2}$ yield
\begin{align}
|U_{t,s}|\gtrsim&\int_{\left\{z_1\in\mathbb D:|1-sz_1|<2\right\}}{|1-sz_1|^{-2}}\lambda^{-1}dV(z_1)\nonumber\\=
&\int_{\mathbb D}{|1-sz_1|^{-2}}\lambda^{-1}dV(z_1)\approx
\frac{1}{\lambda}\log\left(\frac{1}{1-s^2}\right).
\end{align}
Thus \begin{align*}&|\{(z_1,z_2)\in \mathbb D^2:|P(f_s)(z_1,z_2)|>\lambda\}|\gtrsim \frac{1}{\lambda}\log\left(\frac{1}{1-s}\right)=\frac{(-\log {(1-s)})\|f_s\|_{L^1}}{\pi^2\lambda}.\end{align*} Since $(-\log (1-s))$ approaches $\infty$ as $s$ tends to $1$, (\ref{3.1}) holds and the proof is complete.
\end{proof}
\begin{rmk} Using the same example in the proof of Theorem 3.1, one can also show that Theorem 3.1 holds true for the polydisc case.\end{rmk}

The positive result for the weak-type estimate of the Bergman projection is a consequence of the following two theorems from \cite{DHZZ}. 
\begin{thm}[{\hspace{1sp}\cite[Theorem 1.1]{DHZZ}}]
	The Bergman projection is of weak-type $(1,1)$ on the unit disc.
\end{thm}
\begin{thm}[{\hspace{1sp}\cite[Theorem 1.3]{DHZZ}}]
	The Bergman projection is a bounded operator from $L\log^+L(\mathbb D)$ to $L^1(\mathbb D)$.
\end{thm}

\begin{thm}
The Bergman projection on the bidisc $\mathbb D^2$ is of weak-type $L\log^+ L$. 
\end{thm}
\begin{proof}
Let $P_1$ and $P_2$ denote the Bergman projection in variable $z_1$ and $z_2$ respectively. Then the Bergman projection $P$ on the bidisc $\mathbb D^2$ equals $P_1\circ P_2$. By Theorems 3.3, 3.4 and Fubini's theorem, we have for all $f\in L^1(\mathbb D^2)$ 
\begin{align*}
&|\{(z_1,z_2)\in\mathbb D^2:|P(f)(z_1,z_2)|>\lambda\}|\\=&|\{(z_1,z_2)\in\mathbb D^2:|P_1\circ P_2(f)(z_1,z_2)|>\lambda\}|\\=&\int_{\mathbb D}|\{z_1\in\mathbb D:|P_1\circ P_2(f)(z_1,z_2)|>\lambda\}|d V(z_2)\lesssim\frac{\|P_2(f)\|_{L^1}}{\lambda}\lesssim\frac{\|f\|_{L\log^+L}}{\lambda}.
\end{align*}
\end{proof}
By slightly modifying of the proof of Theorem 3.4, one also obtains the following theorem:
\begin{thm}
For $k\in \mathbb N$, the Bergman projection is bounded from $L(\log^+L)^{k+1}(\mathbb D)$ to $L(\log^+L)^{k}(\mathbb D)$.
\end{thm}
\begin{proof} It suffices to show that the Bergman projection is bounded from the unit sphere of  $L(\log^+L)^{k+1}(\mathbb D)$ to $L(\log^+L)^{k}(\mathbb D)$.
Given $f\in L(\log^+L)^{k+1}(\mathbb D)$ with $\|f\|_{L(\log^+L)^{k+1}(\mathbb D)}$ equal to $1$, the definition of the Orlicz norm  $\|\cdot\|_{L(\log^+L)^{k+1}(\mathbb D)}$ implies:
$$1=\|f\|_{L(\log^+L)^{k+1}(\mathbb D)}=\int_{\mathbb D}|f(z)|\left(\log^+|f(z)|\right)^{k+1}dV(z).$$
If $\|P(f)\|_{L(\log^+L)^k(\mathbb D)}\leq1$, then $\|P(f)\|_{L(\log^+L)^k(\mathbb D)}\leq \|f\|_{L(\log^+L)^{k+1}(\mathbb D)}$ and the theorem is proved. 
We turn to consider the case when $\|P(f)\|_{L(\log^+L)^k(\mathbb D)}=\lambda>1$. We show that in this case, the estimate $\|P(f)\|_{L(\log^+L)^k(\mathbb D)}\lesssim\|f\|_{L(\log^+L)^{k+1}(\mathbb D)}$ still holds. 

For a fixed $t>0$, we set $$f_1(z)=\begin{cases}
f(z)& |f(z)|>t\\0& \text{otherwise}
\end{cases},\;\; \text{ and }\;\;  f_2(z)=\begin{cases}
0& |f(z)|>t\\f(z)& \text{otherwise.}
\end{cases}$$ Then $f(z)=f_1(z)+f_2(z)$. For a function $g$ on $\mathbb D$ and a fixed $t>0$, let $g_*(t)$ denote the distribution function: $g_*(t):=|\{z\in \mathbb D:g(z)>t\}|.$
Since the Bergman projection is $L^2$ bounded by its definition and of weak-type (1,1) by Theorem 3.3, we have
\begin{align*}
&(P(f))_*(t)=|\{z\in\mathbb D:|P(f)(z)|>t\}|\\\leq& |\{z\in\mathbb D:|P(f_1)(z)|>\frac{t}{2}\}|+ |\{z\in\mathbb D:|P(f_2)(z)|>\frac{t}{2}\}|\\\lesssim&\frac{\|f_1\|_{L^1}}{t}+\frac{\|f_2\|^2_{L^2}}{t^2}
=\frac{\int_t^\infty f_*(s)ds}{t}+\frac{2\int_0^t sf_*(s)ds}{t^2}.
\end{align*}
Multiplying both sides of the inequality by $(\log t)^{k}+k(\log t)^{k-1}$ and integrating them from 1 to $\infty$ yields:
\begin{align}\label{3.9}
&\int_1^\infty(P(f))_*(t)((\log t)^{k}+k(\log t)^{k-1})dt\nonumber\\\lesssim&\int_1^\infty((\log t)^{k}+k(\log t)^{k-1})\left(\frac{\int_t^\infty f_*(s)ds}{t}+\frac{2\int_0^t sf_*(s)ds}{t^2}\right)dt\nonumber\\=&\int_1^\infty(\frac{(\log s)^{k+1}}{k+1}+{(\log s)^k})f_*(s)ds+2\int_0^\infty\int_{\max\{1,s\}}^\infty\frac{(\log t)^{k}+k(\log t)^{k-1}}{t^2}sf_*(s)dtds\nonumber\\\lesssim&\int_1^\infty(\frac{(\log s)^{k+1}}{k+1}+{(\log s)^k})f_*(s)ds+\int_0^\infty\int_{\max\{1,s\}}^\infty\frac{(\log t)^{k}+1}{t^2}sf_*(s)dtds.
\end{align}
Since $\int_{\max\{1,s\}}^\infty\frac{(\log t)^{k}+1}{t^2}dt=-(\frac{1}{t}+\sum_{j=0}^{k}\frac{k!}{j!t}(\log t)^j)|_{\max\{1,s\}}^\infty$, there holds
$$\int_0^\infty\int_{\max\{1,s\}}^\infty\frac{(\log t)^{k}+1}{t^2}sf_*(s)dtds\lesssim\int_{\mathbb D}|f(z)|(\log^+|f(z)|)^{k}dV(z).$$
We claim that $\int_{\mathbb D}|f(z)|(\log^+|f(z)|)^{k}dV(z)\lesssim 1$. Assume the claim is true.  Then applying this claim and the fact that  $\int_1^\infty(\frac{(\log s)^{k+1}}{k+1}+{(\log s)^k})f_*(s)ds=\|f\|_{L(\log^+ L)^{k+1}}$ into (\ref{3.9}) yields the estimate:
\begin{align}\label{3.10}\int_1^\infty(P(f))_*(t)((\log t)^{k}+k(\log t)^{k-1})dt\lesssim \|f\|_{L(\log^+ L)^{k+1}}.\end{align}
Since $\|P(f)\|_{L(\log^+L)^k}=\lambda$, there holds that
$$\int_{\mathbb D}\left|P\left(\frac{f}{\lambda}\right)(z)\right|\left(\log^+\left|P\left(\frac{f}{\lambda}\right)(z)\right|\right)^kdV(z)=1.$$
Thus $$\int_{\mathbb D}|P({f})(z)|(\log^+|P(\frac{f}{\lambda})(z)|)^kdV(z)=\lambda=\|P(f)\|_{L(\log^+L)^k(\mathbb D)}.$$ Note also that $\lambda>1$. Therefore
\begin{align}\label{3.11}\|P(f)\|_{L(\log^+L)^k(\mathbb D)}&=\int_{\mathbb D}|P({f})(z)|(\log^+|P(\frac{f}{\lambda})(z)|)^kdV(z)\nonumber\\&\leq
\int_{\mathbb D}|P({f})(z)|(\log^+|P({f})(z)|)^kdV(z)\nonumber\\&=\int_1^\infty(P(f))_*(t)((\log t)^{k}+k(\log t)^{k-1})dt.\end{align}
Combining inequalities (\ref{3.10}) and (\ref{3.11}) yields the desired estimate:
$$\|P(f)\|_{L(\log^+L)^k(\mathbb D)}\lesssim \|f\|_{L(\log^+L)^{k+1}}.$$
We turn to prove the claim. By H\"older's inequality, there holds:
\begin{align*}\int_{\mathbb D}|f(z)|(\log^+|f(z)|)^{k}dV(z)&\leq\left(\int_{\mathbb D}|f(z)|(\log^+|f(z)|)^{k+1}dV(z)\right)^{\frac{k}{k+1}}\left(\int_{\mathbb D}|f(z)|dV(z)\right)^{\frac{1}{k+1}}\nonumber\\&=\left(\int_{\mathbb D}|f(z)|dV(z)\right)^{\frac{1}{k+1}}.\end{align*}
Since \begin{align*}\int_{\mathbb D}|f(z)|dV(z)&=\int_{\{z\in \mathbb D:|f(z)<e|\}}|f(z)|dV(z)+\int_{\{z\in \mathbb D:|f(z)\geq e|\}}|f(z)|dV(z)\\&\leq e|\mathbb D|+\int_{\{z\in \mathbb D:|f(z)\geq e|\}}|f(z)|(\log^+|f|)^{k+1}dV(z)\leq e\pi+1\approx 1,\end{align*} there holds
$\int_{\mathbb D}|f(z)|(\log^+|f(z)|)^{k}dV(z)\lesssim 1$ and the proof is complete.
\end{proof}
Theorem 3.6 together with the argument in the proof of Theorem 3.5 then gives the weak-type estimate for the Bergman projection on the polydisc:
\begin{thm}
	The Bergman projection on the polydisc $\mathbb D^n$ is of weak-type $L(\log^+L)^{n-1}$. 
\end{thm}

\section{Weak-type estimates for the Bergman projection on $\mathbb H$}
\begin{thm}
	The Bergman projection on the Hartogs triangle is not of weak-type $(\frac{4}{3},\frac{4}{3})$.
\end{thm} 
\begin{proof}
For a constant $p>\frac{4}{3}$, let $p^\prime=\frac{p}{p-1}$ denote its conjugate index. Set $f_p(z)=\bar z_2|z_2|^{-p^\prime}$. Then 
\begin{align}\|f_p\|^{\frac{4}{3}}_{L^{\frac{4}{3}}}=&\int_{\mathbb H}|z_2|^{\frac{4}{3}(1-p^\prime)}dV(z_1,z_2)=\int_{\mathbb D^2}|z_2|^{2+\frac{4}{3}(1-p^\prime)}dV(z_1,z_2)=\frac{\pi^2(p-1)}{4(p-\frac{4}{3})}.\end{align}
Given $(z_1,z_2)\in \mathbb H$, 
\begin{align}
|P(f_p)(z_1,z_2)|=\int_{\mathbb H}\sum_{a+b\geq -1,a\geq 0}\frac{(z_1\bar w_1)^a(z_2\bar w_2)^b}{\|w_1^aw_2^b\|^2_{L^2}}\bar w_2|w_2|^{-p^\prime}dV(w_1,w_2).
\end{align}
Since the Hartogs triangle is a Reinhardt domain,  it's easy to check using polar coordinates that
$\int_{\mathbb H} \bar w_1^a\bar w_2^b\bar w_2|w_2|^{1-p^\prime}dV(w_1,w_2)\neq 0$
if and only if $a=0$ and $b=-1$. Thus 
\begin{align}
|P(f_p)(z_1,z_2)|&=\left|\int_{\mathbb H}\frac{1}{z_2\bar w_2\|w_2^{-1}\|^2_{L^2}}\bar w_2|w_2|^{-p^\prime}dV(w_1,w_2)\right|\nonumber\\&=\pi^{-2}\left|\int_{\mathbb H}\frac{1}{z_2}|w_2|^{-p^\prime}dV(w_1,w_2)\right|=\frac{p-1}{(3p-{4})|z_2|}.
\end{align}
Note that for $\frac{p-1}{(3p-{4})\lambda}<1$,
\begin{align}
|\{(z_1,z_2)\in\mathbb H:|P(f_p)(z_1,z_2)|>\lambda\}|=&\int_{\left\{(z_1,z_2)\in\mathbb H:|z_2|<\frac{p-1}{(3p-4)\lambda}\right\}}dV(z_1,z_2)
\nonumber\\=&\int_{\left\{(z_1,z_2)\in\mathbb D^2:|z_2|<\frac{p-1}{(3p-{4})\lambda}\right\}}|z_2|^2dV(z_1,z_2)\nonumber\\=&\frac{\pi^2}{4}\left(\frac{p-1}{(3p-{4})\lambda}\right)^4\approx\frac{\|f_p\|^{{4}/{3}}_{L^{{4}/{3}}}}{\lambda^{{4}/{3}}}\frac{1}{(p-{4}/{3})^3\lambda^{{8}/{3}}}.
\end{align}
Setting $p={4}/{3}+{\lambda^{-{9}/{10}}}$, then ${(p-1)}(3p-{4})^{-1}\lambda^{-1}$ still goes to 0 as $\lambda$ tends to $\infty$. Hence ${(p-1)}(3p-{4})^{-1}\lambda^{-1}<1$ holds. On the other hand, 
${(p-{4}/{3})^{-3}\lambda^{-{8}/{3}}}=\lambda^{{1}/{30}},$
which is blowing up as $\lambda$ tends to $\infty$. Therefore, the weak-type estimate 
$$|\{(z_1,z_2)\in\mathbb H:|P(f_p)(z_1,z_2)|>\lambda\}|\lesssim \frac{ \|f_p\|^{{4}/{3}}_{L^{{4}/{3}}}}{\lambda^{{4}/{3}}}$$
fails and the Bergman projection on $\mathbb H$ is not of weak-type $({4}/{3},{4}/{3})$.
\end{proof}

\begin{thm}
	The Bergman projection on the Hartogs triangle is of weak-type $({4},{4})$.
\end{thm} 
\begin{proof}
Let $f$ be an arbitrary function in $L^4(\mathbb H)$. Then
\begin{align}
\|f\|^4_{L^4}=&\int_{\mathbb H}|f(z_1,z_2)|^4dV(z_1)dV(z_2)\nonumber\\=&\int_{\mathbb D^2}|f(z_1z_2,z_2)|^4|z_2|^2dV(z_1)dV(z_2)\nonumber\\=&\int_{\mathbb D^2}|z_2f(z_1z_2,z_2)|^4|z_2|^{-2}dV(z_1)dV(z_2).\end{align}
Set $g(z_1,z_2)=z_2f(z_1z_2,z_2)$. Then $g\in L^4(\mathbb D^2,|z_2|^{-2}dV)$ and $\|g\|_{L^4(\mathbb D^2,|z_2|^{-2}dV)}=\|f\|_{L^4(\mathbb H)}.$ 
Note that
\begin{align}
|P(f)(z_1,z_2)|&=\left|\int_{\mathbb H}\frac{f(w_1,w_2)}{\pi^2z_2\bar w_2(1-\frac{z_1\bar w_1}{z_2 \bar w_2})^2(1-z_2\bar w_2)^2}dV(w_1,w_2)\right|\nonumber\\&=\left|\int_{\mathbb D^2}\frac{w_2f(w_1w_2,w_2)}{\pi^2z_2(1-\frac{z_1}{z_2}\bar w_1)^2(1-z_2\bar w_2)^2}dV(w_1,w_2)\right|=\frac{1}{|z_2|}\left|P_{\mathbb D^2}(g)\left(\frac{z_1}{z_2},z_2\right)\right|.
\end{align}
Then there holds
\begin{align}\label{4.8}
&|\{(z_1,z_2)\in\mathbb H:|P(f)(z_1,z_2)|>\lambda\}|\nonumber\\=&\int_{\left\{(z_1,z_2)\in\mathbb H: \frac{1}{|z_2|}\left|P_{\mathbb D^2}(g)\left(\frac{z_1}{z_2},z_2\right)\right|>\lambda\right\}}dV(z_1,z_2)
\nonumber\\=&\int_{\left\{(z_1,z_2)\in\mathbb D^2: \frac{1}{|z_2|}\left|P_{\mathbb D^2}(g)\left({z_1},z_2\right)\right|>\lambda\right\}}|z_2|^2dV(z_1,z_2)
\nonumber\\=&\int_{\left\{(z_1,z_2)\in\mathbb D^2: |z_2|\leq\frac{1}{2} \text{ and }\frac{1}{|z_2|}\left|P_{\mathbb D^2}(g)\left({z_1},z_2\right)\right|>\lambda\right\}}|z_2|^2dV(z_1,z_2)\nonumber\\&+\int_{\left\{(z_1,z_2)\in\mathbb D^2: |z_2|>\frac{1}{2} \text{ and } \frac{1}{|z_2|}\left|P_{\mathbb D^2}(g)\left({z_1},z_2\right)\right|>\lambda\right\}}|z_2|^2dV(z_1,z_2).
\end{align}
We first consider the integral with $|z_2|\leq\frac{1}{2}$.
When $|z_2|\leq\frac{1}{2}$, the Bergman kernel $K_{\mathbb D^2}$ satisfies:
\begin{align}|K_{\mathbb D^2}(z_1,z_2;\bar w_1,\bar w_2)|=\frac{1}{\pi^2|1-z_1\bar w_1|^2|1-z_2\bar w_2|^2}\approx\frac{1}{\pi|1-z_1\bar w_1|^2}=|K_{\mathbb D}(z_1;\bar w_1)|.\end{align}
Therefore,
\begin{align}
|P_{\mathbb D^2}(g)\left({z_1},z_2\right)|&=\left|\int_{\mathbb D^2}\frac{g(w_1,w_2)}{\pi^2(1-z_1\bar w_1)^2(1-z_2\bar w_2)^2}dV(w_1,w_2)\right|\nonumber\\&\leq\int_{\mathbb D^2}\frac{|g(w_1,w_2)|}{\pi^2|1-z_1\bar w_1|^2|1-z_2\bar w_2|^2}dV(w_1,w_2)
\nonumber\\&\lesssim\int_{\mathbb D}\frac{\int_{\mathbb D}|g(w_1,w_2)|dV(w_2)}{\pi|1-z_1\bar w_1|^2}dV(w_1)\nonumber\\&=P^+_{\mathbb D}\left(\int_{\mathbb D}|g(w_1,w_2)|dV(w_2)\right)(z_1).
\end{align}
Set $G(w_1)=\int_{\mathbb D}|g(w_1,w_2)|dV(w_2)$. Then there exists a constant $C$ such that,
\begin{align}\label{4.101}
&\int_{\left\{(z_1,z_2)\in\mathbb D^2: |z_2|\leq\frac{1}{2} \text{ and }\frac{1}{|z_2|}\left|P_{\mathbb D^2}(g)\left({z_1},z_2\right)\right|>\lambda\right\}}|z_2|^2dV(z_1,z_2)\nonumber\\\leq&\int_{\left\{(z_1,z_2)\in\mathbb D^2: |z_2|\leq\frac{1}{2} \text{ and }\frac{1}{|z_2|}P^+_{\mathbb D}(G)\left({z_1}\right)>C\lambda\right\}}|z_2|^2dV(z_1,z_2)
\nonumber\\\leq&\int_{\mathbb D}\int_{\left\{z_2\in\mathbb D: |z_2|\leq\frac{1}{2} \text{ and }\frac{1}{C\lambda}P^+_{\mathbb D}(G)\left({z_1}\right)>|z_2|\right\}}|z_2|^2dV(z_2)d(z_1)
\nonumber\\\lesssim&\int_{\mathbb D}\int_{0}^{\frac{1}{C\lambda}P^+_{\mathbb D}(G)\left({z_1}\right)}r^3drd(z_1)
\lesssim\int_{\mathbb D}\frac{\left(P^+_{\mathbb D}(G)\left({z_1}\right)\right)^4}{\lambda^4}dV(z_1)=\frac{\|P^+_{\mathbb D}(G)\|^4_{L^4(\mathbb D)}}{\lambda^4}.
\end{align}
 Applying H\"older's inequality,
\begin{align}\label{4.11}
\int_{\mathbb D}G(w_1)^4dV(w_1)\lesssim \int_{\mathbb D}\int_{\mathbb D}|g(w_1,w_2)|^4dV(w_2)dV(w_1)=\|g\|^4_{L^4(\mathbb D^2)}\leq\|g\|^4_{L^4(\mathbb D^2,|w_2|^{-2})}.\end{align}
Thus $G(w_1)$ is in $L^4(\mathbb D)$. By (\ref{4.11}) and  the $L^p$ boundedness of $P^+_{\mathbb D}$, inequality (\ref{4.101}) yields
\begin{align}\label{4.12}
\int_{\left\{(z_1,z_2)\in\mathbb D^2: |z_2|\leq\frac{1}{2} \text{ and }\frac{1}{|z_2|}\left|P_{\mathbb D^2}(g)\left({z_1},z_2\right)\right|>\lambda\right\}}|z_2|^2dV(z_1,z_2)\lesssim\frac{\|g\|^4_{L^4( \mathbb D^2,|w_2|^{-2})}}{\lambda^4}=\frac{\|f\|^4_{L^4( \mathbb H)}}{\lambda^4}.
\end{align}
Now we turn to the integral in (\ref{4.8}) with $|z_2|> \frac{1}{2}$. By $|z_2|> \frac{1}{2}$, there holds $\frac{1}{|z_2|}<2$ and 
\begin{align}\label{4.13}
&\int_{\left\{(z_1,z_2)\in\mathbb D^2: |z_2|>\frac{1}{2} \text{ and } \frac{1}{|z_2|}\left|P_{\mathbb D^2}(g)\left({z_1},z_2\right)\right|>\lambda\right\}}|z_2|^2dV(z_1,z_2)\nonumber\\
\leq&\int_{\left\{(z_1,z_2)\in\mathbb D^2: \left|P_{\mathbb D^2}(g)\left({z_1},z_2\right)\right|>\frac{\lambda}{2}\right\}}|z_2|^2dV(z_1,z_2)
\nonumber\\
\leq&\frac{2^4\int_{\mathbb D^2} \left|P_{\mathbb D^2}(g)\left({z_1},z_2\right)\right|^4|z_2|^2dV(z_1,z_2)}{\lambda^4}
\nonumber\\
\leq&\frac{2^4\int_{\mathbb D^2} \left|P_{\mathbb D^2}(g)\left({z_1},z_2\right)\right|^4dV(z_1,z_2)}{\lambda^4}.
\end{align}
Since $P_{\mathbb D^2}$ is also $L^p$ bounded  for $1<p<\infty$, there holds
$$\frac{2^4\int_{\mathbb D^2} \left|P_{\mathbb D^2}(g)\left({z_1},z_2\right)\right|^4dV(z_1,z_2)}{\lambda^4}\lesssim\frac{\|g\|^4_{L^4( \mathbb D^2)}}{\lambda^4}\leq\frac{\|g\|^4_{L^4(\mathbb D^2,|w_2|^{-2})}}{\lambda^4}=\frac{\|f\|^4_{L^4( \mathbb H)}}{\lambda^4}.$$
Hence we also obtain the inequality 
\begin{align}\label{4.14}
&\int_{\left\{(z_1,z_2)\in\mathbb D^2: |z_2|>\frac{1}{2} \text{ and } \frac{1}{|z_2|}\left|P_{\mathbb D^2}(g)\left({z_1},z_2\right)\right|>\lambda\right\}}|z_2|^2dV(z_1,z_2)\lesssim\frac{\|f\|^4_{L^4( \mathbb H)}}{\lambda^4}.
\end{align}
Applying (\ref{4.12}) and (\ref{4.14}) to (\ref{4.8}) yields the desired weak-type (4,4) estimate.
\end{proof}
\begin{rmk} It can be shown that if the Bergman projection $P$ on a weighted space $L^p(\mathbb D^2,\mu)$ is of weak-type $(p,p)$, then $P$ is bounded on $L^p(\mathbb D^2,\mu)$. The idea of the proof can be found  in Theorem 1 in \cite{Rahm} and Theorem 1.2 in \cite{ZhenghuiWick2}. Theorem 4.2, on the other hand, shows a different phenomenon in the Hartogs triangle case: the Bergman projection on $\mathbb H$ is of weak-type $(4,4)$ but not $L^4$-bounded. This difference is caused by the fact that  while $L^{4}(\mathbb D^2,|z_2|^{-2})$ and $L^4(\mathbb H)$ are isometrically equivalent via the mapping $(z_1,z_2)\to (z_1z_2,z_2)$ between $\mathbb D\times \mathbb D\backslash\{0\}$ and $\mathbb H$,  the weak spaces $L^{4,\infty}(\mathbb D^2,|z_2|^{-2})$ and $L^{4,\infty}(\mathbb H)$ are not.\end{rmk}

\begin{rmk}
	Since the Bergman projection $P$ is self-adjoint, a duality argument together with Theorem 4.2 implies that $P$  is bounded from the Lorentz space $L^{4/3,1}(\mathbb H)$ to $L^{{4}/{3}}(\mathbb H)$. See for example Theorem 1.4.16 in \cite{Grafakos}.
\end{rmk}

\begin{rmk} Theorems 4.1 and 4.2 also provide an alternative proof of the $L^p$-regularity result for the Bergman projection on the Hartogs triangle: by interpolation, the weak type $(4,4)$ and $L^2$ regularity of the Bergman projection implies that the projection is $L^p$ bounded for $p\in [2,4)$. Then a duality argument yields the $L^p$ regularity for $p\in ({4}/{3},4)$. Since the projection is not of weak-type $({4}/{3},{4}/{3})$, it's not $L^{{4}/{3}}$ bounded, and hence not $L^p$ bounded for $p\notin ({4}/{3},4)$. Therefore the Bergman projection on $\mathbb H$ is $L^p$ bounded if and only if $p\in ({4}/{3},4)$.\end{rmk}

Using the same idea of the proof of Theorem 4.2, one can obtain the following weak-type estimate for $P$ near $L^{\frac{4}{3}}$. Here we provide another proof using Lemma 2.2.
\begin{thm}
	For any $\epsilon>0$, the Bergman projection $P$ on the Hartogs triangle satisfies the following weak-type estimate:
	\begin{equation}\label{4.16}
	|\{(z_1,z_2)\in\mathbb H:|P(f)(z_1,z_2)|>\lambda\}|\lesssim \frac{\|f\|^{{4}/{3}}_{L^{{4}/{3}}(\mathbb H,|z_2|^{-\epsilon})}}{\lambda^{{4}/{3}}}.
	\end{equation}
\end{thm} 
\begin{proof}We claim that the Bergman projection is bounded on $L^{\frac{4}{3}}(\mathbb H,|z_2|^{-\epsilon})$. Then the desired estimate holds:
	\begin{align}
		|\{(z_1,z_2)\in\mathbb H:|P(f)(z_1,z_2)|>\lambda\}|\leq&
		\int_{\{(z_1,z_2)\in\mathbb H:|P(f)(z_1,z_2)|>\lambda\}}|z_2|^{-\epsilon}dV(z_1,z_2)\nonumber\\
		\leq&\int_{\mathbb H}\frac{|P(f)(z_1,z_2)|^{{4}/{3}}}{\lambda^{{4}/{3}}}|z_2|^{-\epsilon}dV(z_1,z_2)\nonumber\\\lesssim&\frac{\|f\|^{{4}/{3}}_{L^{{4}/{3}}(\mathbb H,|z_2|^{-\epsilon})}}{\lambda^{{4}/{3}}}.
	\end{align}
	To prove the claim, we recall that for any given $f\in L^{\frac{4}{3}}(\mathbb H,|z_2|^{-\epsilon})$, the induced function  $g(w_1,w_2):=w_2f(w_1w_2,w_2)$ is in $L^{\frac{4}{3}}(\mathbb D^2,|z_2|^{\frac{2}{3}-\epsilon})$. Moreover,
	\begin{align}
	|P(f)(z_1,z_2)|&=\left|\int_{\mathbb H}\frac{f(w_1,w_2)}{\pi^2z_2\bar w_2(1-\frac{z_1\bar w_1}{z_2 \bar w_2})^2(1-z_2\bar w_2)^2}dV(w_1,w_2)\right|\nonumber\\&=\left|\int_{\mathbb D^2}\frac{w_2f(w_1w_2,w_2)}{\pi^2z_2(1-\frac{z_1}{z_2}\bar w_1)^2(1-z_2\bar w_2)^2}dV(w_1,w_2)\right|=\frac{1}{|z_2|}\left|P_{\mathbb D^2}(g)\left(\frac{z_1}{z_2},z_2\right)\right|.\nonumber
	\end{align}
	Then it is easy to see that the two operator norms $\|P\|_{L^{\frac{4}{3}}(\mathbb H,|z_2|^{-\epsilon})}$ and $\|P_{\mathbb D^2}\|_{L^{\frac{4}{3}}(\mathbb D^2,|z_2|^{\frac{2}{3}-\epsilon})}$ are identical.
We first show that $P_{\mathbb D}$ is bounded on $L^{\frac{4}{3}}(\mathbb D,|w|^{\frac{2}{3}-\epsilon})$. Recall the Carleson tent $T_z$ in Lemma 2.2. When $|z|>\frac{1}{2}$, the function $|w|\approx 1$ for all $w\in T_z$. Hence for $|z|>\frac{1}{2}$, we have
	\begin{equation}\label{4.261}
	\frac{\int_{T_z}|w|^{\frac{2}{3}-\epsilon}dV(w)\left(\int_{ T_z}|w|^{(\frac{2}{3}-\epsilon)\frac{1}{1-p}}dV(w)\right)^{p-1}}{\left(V(T_z)\right)^{p}}\lesssim 1.
	\end{equation}
	For $|z|\leq \frac{1}{2}$, the Lebesgue measure of $|T_z|\approx 1$. Thus for $p=\frac{4}{3}$, there holds
	
	\begin{align}\label{4.27}&\frac{\int_{T_z}|w|^{\frac{2}{3}-\epsilon}dV(w)\left(\int_{ T_z}|w|^{(\frac{2}{3}-\epsilon)\frac{1}{1-p}}dV(w)\right)^{p-1}}{\left|T_z\right|^{p}}\nonumber\\\lesssim& \int_{\mathbb D}|w|^{\frac{2}{3}-\epsilon}dV(w)\left(\int_{\mathbb D}|w|^{(\frac{2}{3}-\epsilon)\frac{1}{1-p}}dV(w)\right)^{p-1}<\infty.\end{align}
 Combining (\ref{4.261}) and (\ref{4.27}) yields that the Bekoll\'e-Bonami constant $B_{\frac{4}{3}}(|w|^{\frac{2}{3}-\epsilon})$ is finite. Then Lemma 2.2 implies the boundedness of $P_{\mathbb D}$ on $L^{\frac{4}{3}}(\mathbb D,|w|^{\frac{2}{3}-\epsilon})$. Note that $P_{\mathbb D^2}=P_2\circ P_1$ where $P_j$ is the projection in the variable $z_j$. There holds via Fubini's theorem that
 \begin{align}
 \|P_{\mathbb D^2}(g)\|^{\frac{4}{3}}_{L^{\frac{4}{3}}(\mathbb D^2,|z_2|^{\frac{2}{3}-\epsilon})}=&\int_{\mathbb D^2}|P_{\mathbb D^2}(g)(z_1,z_2)|^{\frac{4}{3}}|z_2|^{\frac{2}{3}-\epsilon}dV(z_1,z_2)\nonumber\\=&\int_{\mathbb D^2}|P_2P_1(g)(z_1,z_2)|^{\frac{4}{3}}|z_2|^{\frac{2}{3}-\epsilon}dV(z_1,z_2)\nonumber\\=&\int_{\mathbb D}\|P_2P_1(g)(z_1,\cdot)\|^{\frac{4}{3}}_{L^{\frac{4}{3}}(\mathbb D,|z|^{\frac{2}{3}-\epsilon})}dV(z_1)\nonumber\\\lesssim&\int_{\mathbb D}\|P_1(g)(z_1,\cdot)\|^{\frac{4}{3}}_{L^{\frac{4}{3}}(\mathbb D,|z|^{\frac{2}{3}-\epsilon})}dV(z_1)=\|P_1(g)\|^{\frac{4}{3}}_{L^{\frac{4}{3}}(\mathbb D^2,|z_2|^{\frac{2}{3}-\epsilon})}.
 \end{align}
 By Fubini's theorem again, $P_1$ is bounded on $L^{\frac{4}{3}}(\mathbb D^2,|z_2|^{\frac{2}{3}-\epsilon})$. Thus
 $$ \|P_{\mathbb D^2}(g)\|^{\frac{4}{3}}_{L^{\frac{4}{3}}(\mathbb D^2,|z_2|^{\frac{2}{3}-\epsilon})}\lesssim\|P_1(g)\|^{\frac{4}{3}}_{L^{\frac{4}{3}}(\mathbb D^2,|z_2|^{\frac{2}{3}-\epsilon})}\lesssim \|g\|^{\frac{4}{3}}_{L^{\frac{4}{3}}(\mathbb D^2,|z_2|^{\frac{2}{3}-\epsilon})}.$$
 The boundedness of  $P_{\mathbb D^2}$ on ${L^{\frac{4}{3}}(\mathbb D^2,|z_2|^{\frac{2}{3}-\epsilon})}$ then implies the boundedness of $P$ on ${L^{\frac{4}{3}}(\mathbb H,|z_2|^{-\epsilon})}$, which completes the proof of the claim.
\end{proof}
Note that for $\alpha<-1$, the integral
 \begin{align}\label{4.21}\int_{\mathbb D}|z|^{-2}(-\log|z|+1)^{\alpha}dV(z)=&2\pi\int_0^1r^{-1}(-\log r+1)^\alpha dr\nonumber\\=&2\pi\int_{-\infty}^0(-t+1)^\alpha dt=\frac{-2\pi}{\alpha+1}<\infty.\end{align} An similar argument as in inequalities (\ref{4.261}) and (\ref{4.27}) then implies that the Bekoll\'e-Bonami constant $B_{\frac{4}{3}}(|z|^{\frac{2}{3}}(-\log|z|+1)^{\epsilon})$ is finite for $\epsilon>\frac{1}{3}$. Since $x^{-a}\gtrsim -\log x+1\geq 1$ on $(0,1)$ for all $a>0$, replacing $|z_2|^{-\epsilon}$ by $(-\log|z_2|+1)^{\epsilon}$ in the proof of Theorem 4.6 yields a better estimate:
\begin{thm}
	For any $\epsilon>\frac{1}{3}$, the Bergman projection $P$ on the Hartogs triangle satisfies the following weak-type estimate:
	\begin{equation}\label{4.22}
	|\{(z_1,z_2)\in\mathbb H:|P(f)(z_1,z_2)|>\lambda\}|\lesssim \frac{\|f\|^{{4}/{3}}_{L^{{4}/{3}}(\mathbb H,(-\log|z_2|+1)^{\epsilon})}}{\lambda^{{4}/{3}}}.
	\end{equation}
\end{thm} 
 For $\epsilon\leq\frac{1}{3}$, Theorem 4.7 does not hold anymore. Below we provide an example to illustrate the failure of the estimate (\ref{4.22}) when $\epsilon=\frac{1}{3}$. 

 Set $f_p(z)=\bar z_2|z_2|^{-p^\prime}(-\log|z_2|+1)^{-1}$. Then 
\begin{align}\label{4.23}
\|f_p\|^{{4}/{3}}_{L^{{4}/{3}}(\mathbb H,(-\log|z_2|+1)^{{1}/{3}})}=&\int_{\mathbb H}|z_2|^{\frac{4}{3}(1-p^\prime)}(-\log|z_2|+1)^{-1}dV(z_1,z_2)\nonumber\\=&\int_{\mathbb D^2}|z_2|^{2+\frac{4}{3}(1-p^\prime)}(-\log|z_2|+1)^{-1}dV(z_1,z_2)\nonumber\\=&2\pi^2\int_{0}^1x^{3+\frac{4}{3}(1-p^\prime)}(-\log x+1)^{-1}dx\nonumber\\=&2\pi^2 e^{4+\frac{4}{3}(1-p^\prime)}\text{E}_1(4+\frac{4}{3}(1-p^\prime)),
\end{align}
where $E_1$ is  the \emph{exponential integral} defined by $$\text{E}_1(x)=\int_{x}^{\infty}t^{-1}e^{-t}d t.$$ Note that (\hspace{1sp}\cite{Abramowitz}, p. 229, 5.1.20)
$$\frac{1}{2}e^{-x}\log\left(1+\frac{2}{x}\right)<\text{E}_1(x)<e^{-x}\log\left(1+\frac{1}{x}\right).$$
Therefore as $p\to \frac{4}{3}$, there holds $4+\frac{4}{3}(1-p^\prime)\to 0$ and 
\begin{align}
\text{E}_1\left(4+\frac{4}{3}(1-p^\prime)\right)\approx \log\left(\frac{1}{4+\frac{4}{3}(1-p^\prime)}\right)\approx\log\left(\frac{1}{3p-4}\right).
\end{align}
Substituting this back into (\ref{4.23}) yields $\|f_p\|^{{4}/{3}}_{L^{{4}/{3}}(\mathbb H,(-\log|z_2|+1)^{{1}/{3}})}\approx \log\left(\frac{1}{3p-4}\right)$. On the other hand, 
\begin{align}
|P(f_p)(z_1,z_2)|&=\left|\int_{\mathbb H}\frac{1}{z_2\bar w_2\|w_2^{-1}\|^2_{L^2}}\bar w_2|w_2|^{-p^\prime}(-\log|z_2|+1)^{-1}dV(w_1,w_2)\right|\nonumber\\&=\pi^{-2}\left|\int_{\mathbb H}\frac{1}{z_2}|w_2|^{-p^\prime}(-\log|z_2|+1)^{-1}dV(w_1,w_2)\right|\approx\frac{1}{|z_2|}\log\left(\frac{1}{3p-4}\right).
\end{align}
Therefore, when $\log\left(\frac{1}{3p-4}\right)\frac{1}{\lambda}$ is small, we have
\begin{align}
&|\{(z_1,z_2)\in\mathbb H:|P(f_p)(z_1,z_2)|>\lambda\}|\nonumber\\\gtrsim&\int_{\left\{(z_1,z_2)\in\mathbb H:|z_2|<\log\left(\frac{1}{3p-4}\right)\frac{1}{\lambda}\right\}}dV(z_1,z_2)
\nonumber\\=&\int_{\left\{(z_1,z_2)\in\mathbb D^2:|z_2|<\log\left(\frac{1}{3p-4}\right)\frac{1}{\lambda}\right\}}|z_2|^2dV(z_1,z_2)\nonumber\\=&\frac{\pi^2}{4}\left(\log\left(\frac{1}{3p-4}\right)\frac{1}{\lambda}\right)^4\approx\frac{\|f_p\|^{{4}/{3}}_{L^{{4}/{3}}}}{\lambda^{{4}/{3}}}\left(\log\left(\frac{1}{p-{4}/{3}}\right)\right)^{3}\frac{1}{\lambda^{{8}/{3}}}.
\end{align}
Setting $p={4}/{3}+\exp\left\{-\lambda^{{9}/{10}}\right\}$, then $\log\left(\frac{1}{3p-4}\right)\frac{1}{\lambda}$ still goes to 0 as $\lambda$ tends to $\infty$. Hence $\frac{p-1}{(3p-{4})\lambda}<1$ holds. Note that
$\left(\log\left(\frac{1}{p-{4}/{3}}\right)\right)^{3}\frac{1}{\lambda^{{8}/{3}}}=\lambda^{{1}/{30}},$
which is blowing up as $\lambda$ tends to $\infty$. Therefore, the weak-type estimate
$$|\{(z_1,z_2)\in\mathbb H:|P(f_p)(z_1,z_2)|>\lambda\}|\lesssim\frac{\|f\|^{{4}/{3}}_{L^{{4}/{3}}(\mathbb H,(-\log|z_2|+1)^{\epsilon})}}{\lambda^{{4}/{3}}} \;\;\text{ fails for }\epsilon={1}/{3}.$$
\begin{rmk} Estimate (\ref{4.22})  in Theorem 4.7 is a consequence of the finite integral in (\ref{4.21}) and the Bekoll\'e-Bonami theory on the unit disc. The integrand $|z|^{-2}(-\log|z|+1)^\alpha$ blows up at a slower speed near the origin than $|z|^{-2}$ and hence is in $L^1(\mathbb D)$. Similarly, one can construct an integrable function $|z|^{-2}(-\log|z|+1)^{-1}(\log(-\log|z|+1)+1)^{\alpha}$ with $\alpha<-1$ from $|z|^{-2}(-\log|z|+1)^{-1}$. Iterating this process, we obtain a sequence of functions $\{f_{\alpha,j}(z)\}$ in $L^1(\mathbb D)$ where $f_{\alpha,j}(z)=|z|^{-2}h_j^{\alpha}(z)\prod_{k=1}^{j-1}h_k^{-1}(z)$ with $\alpha<-1$ and $h_j(z)$ defined as follows:\begin{enumerate}
		\item $h_1(z)=-\log|z|+1$;
		\item  $h_{j+1}(z)=\log(h_j(z)+1)+1$ for $j>0$.
	\end{enumerate}
Then repeating the argument for (\ref{4.261}) and (\ref{4.27}) yields that the Bekoll\'e-Bonami constant $B_{\frac{4}{3}}(f^{-{1}/{3}}_{\alpha,j})<\infty$. Using this fact, (\ref{4.22}) can be generalized as below:
\begin{thm}
	Let $f_{\alpha,j}$ be defined as above. For any $\alpha<-1$, the Bergman projection $P$ on the Hartogs triangle satisfies the following weak-type estimate:
	\begin{equation}
	|\{(z_1,z_2)\in\mathbb H:|P(f)(z_1,z_2)|>\lambda\}|\lesssim \frac{\|f\|^{{4}/{3}}_{L^{{4}/{3}}\left(\mathbb H,\left(|z_2|^{2}f_{\alpha,j}(z_2)\right)^{-{1}/{3}}\right)}}{\lambda^{{4}/{3}}}.
	\end{equation}
\end{thm}
\end{rmk}
\begin{rmk} Despite Theorems 4.6, 4.7 and 4.9, a sharp weak-type estimate for $P$ near $L^{{4}/{3}}$ is still unknown to us. One of our guesses is the weak-type $L^{{4}/{3}}(\log^+L)^{\alpha}$ estimate.  For $p>{4}/{3}$ and $\alpha>0$, there holds $\|\bar z_2|z_2|^{-p}\|^{{4}/{3}}_{L^{{4}/{3}}(\log^+L)^{\alpha}}\lesssim\left(p-{4}/{3}\right)^{-\alpha-1}.$ Then an argument as in the proof of Theorem 4.1 would  imply that the projection $P$ is not of weak-type $L^{{4}/{3}}(\log^+L)^{\alpha}$ for $\alpha<{1}/{3}$. For $\alpha={1}/{3}$, the estimate holds for $f_p(z)=\bar z_2|z_2|^{-p^\prime}$ which is served as a counterexample in the proof of Theorem 4.1. Hence we suspect that the Bergman projection is bounded from $L^{{4}/{3}}(\log^+L)^{{1}/{3}}(\mathbb H)$ to $L^{{4}/{3},\infty}(\mathbb H)$. We hope to further investigate it in the future.\end{rmk}
\begin{rmk} $L^p$ regularity of the Bergman projection has also been studied on various generalizations of the Hartogs triangle. See for instance \cite{EM, EM2, CHEN}. It would be interesting to study the weak-type regularity of the Bergman projection in those settings.\end{rmk}

\vskip 5pt

\bibliographystyle{alpha}
\bibliography{2}
\end{document}